\documentclass[11pt]{amsart}
\usepackage{amsmath,amssymb,mathrsfs,color}
\usepackage[title]{appendix}
\usepackage{amssymb}
\allowdisplaybreaks

\def\N{{\mathbb N}}

\def\R{{\mathbb R}}
\def\P{{\mathbb P}}
\def\E{{\mathbf E}}

\def\eps{\epsilon}
\newtheorem{thm}{Theorem}[section]

\newtheorem{lem}[thm]{Lemma}
\newtheorem{prop}[thm]{Proposition}

\theoremstyle{definition}
\newtheorem{de}[thm]{Definition}
\theoremstyle{remark}
\newtheorem{rem}[thm]{Remark}
\newtheorem{exam}[thm]{Example}
\numberwithin{equation}{section}

\newcommand{\rmd}{{\rm d}}

\allowdisplaybreaks

\begin{document}

\title[Wasserstein convergence rates in the invariance principle]{Wasserstein convergence rates in the invariance principle for deterministic dynamical systems}

\author{Zhenxin Liu}
\address{Z. Liu: School of Mathematical Sciences, Dalian University of Technology, Dalian
116024, P. R. China}
\email{zxliu@dlut.edu.cn}

\author{Zhe Wang}
\address{Z. Wang: School of Mathematical Sciences, Dalian University of Technology, Dalian
116024, P. R. China}
 \email{zwangmath@hotmail.com; wangz1126@mail.dlut.edu.cn}

\date{December 25, 2022}



\begin{abstract}
In this paper, we consider the convergence rate with respect to Wasserstein distance in the invariance principle for deterministic nonuniformly hyperbolic systems, where both discrete time systems and flows are included. Our results apply to uniformly hyperbolic systems and large classes of nonuniformly hyperbolic systems including intermittent maps, Viana maps, finite horizon planar periodic Lorentz gases and others. Furthermore, as a nontrivial application to homogenization problem, we investigate the $\mathcal{W}_2$-convergence rate of a fast-slow discrete deterministic system to a stochastic differential equation.
\end{abstract}

\maketitle

\section{Introduction}
\setcounter{equation}{0}

It is well known that deterministic dynamical systems can exhibit some statistical limit properties if the system is chaotic enough and the observable satisfies some regularity conditions. In recent years, we have seen growing research interests in statistical limit properties for deterministic systems such as the law of large numbers (or Birkhoff's ergodic theorem), central limit theorem (CLT), weak invariance principle (WIP), almost sure invariance principle (ASIP), large deviations and so on.

The WIP (also known as the functional CLT) states that a stochastic process constructed by the sums of random variables with suitable scale converges weakly to Brownian motion, which is a  far-reaching generalization of the CLT. Donsker's theorem \cite{MR40613} is the prototypical invariance principle, which dealt with independent and identically distributed random variables. Later, different versions were extensively studied. In particular, many authors studied the WIP and ASIP for dynamical systems with some hyperbolicity. Denker and Philipp \cite{MR779712} proved the ASIP for uniformly hyperbolic diffeomorphisms and flows. The results are stronger as the ASIP implies the WIP and CLT. Melbourne and Nicol \cite{MR2175992} investigated the ASIP for nonuniformly expanding maps and nonuniformly hyperbolic diffeomorphisms that can be modelled by a Young tower \cite{MR1637655, MR1750438},  and also obtained corresponding results for flows. After that, there are a lot of works on the WIP for nonuniformly hyperbolic systems, which we will not mention here.

To our best knowledge, there are only two works on rates of convergence in the WIP for deterministic dynamical systems in spite that there are many results on the convergence itself. In his Ph.D thesis \cite{antoniou2018rates}, Antoniou obtained the rate of convergence in L\'{e}vy-Prokhorov distance for uniformly expanding maps using the martingale approximation method and applying an estimate for martingale difference arrays \cite{MR1357802} by Kubilyus. Then following the same method, together with Melbourne, he \cite{MR3975864} further generalized the convergence rates to nonuniformly expanding/hyperbolic systems. Specifically, they apply a new version of martingale-coboundary decomposition \cite{MR3795019} by Korepanov et al.

The Wasserstein distance is used extensively in recent years to metrize weak convergence, and it is stronger and contains more information than the L\'{e}vy-Prokhorov distance since it involves the metric of the underlying space. This distance finds important applications in the fields of optimal transport, geometry, partial differential equations etc; see e.g.\ Villani \cite{MR2459454} for details.  There are some results on Wasserstein convergence rates for the CLT in the community of probability and statistics; see e.g.\ \cite{MR2446294,MR2506123,MR2548505}. However, to our knowledge, there are no related results on the invariance principle for dynamical systems. Motivated by \cite{antoniou2018rates, MR3975864}, we aim to estimate the Wasserstein convergence rate in the WIP for nonuniformly hyperbolic systems.

For discrete time systems, we first consider a martingale as an intermediary process. In \cite{MR3975864}, the authors apply a result of Kubilyus \cite{MR1357802} and the key is to estimate the distance between $W_n$ defined in \eqref{exp} below and the intermediary process. In the present paper, we use the ideas in \cite{MR3975864} to estimate the distance between $W_n$ and the intermediary process. Hence most of our efforts are to deal with the Wasserstain distance between the intermediary process and Brownian motion, which is handled by a martingale version of the Skorokhod embedding theorem. In this way, we obtain the rate of convergence $O(n^{-\frac{1}{4}+\delta})$ in the Wasserstein distance, where $\delta$ depends on the degree of nonuniformity. When the system that can be modelled by a Young tower has a superpolynomial tail, $\delta$ can be arbitrarily small. It is well known that it is much simpler to prove statistical limit properties for discrete time systems than that for continuous time ones. So following the argument in \cite{AOP16, MT04, levy15}, we obtain the convergence rates for flows by reducing continuous time to discrete time.

Our results are applicable to uniformly hyperbolic systems and large classes of nonuniformly hyperbolic systems modelled by a Young tower with superpolynomial and polynomial tails. In comparison with \cite{MR3975864}, when the dynamical system has a superpolynominal tail, we can obtain the same convergence rate $O(n^{-\frac{1}{4}+\delta})$ for $\delta$ arbitrarily small. While in our case, the price to pay is that the dynamical system need to have stronger mixing properties. For example, we consider the Pomeau-Manneville intermittent map \eqref{P-M} which has a polynominal tail. By \cite{MR3975864}, the convergence rate is $O(n^{-\frac{1}{4}+\frac{\gamma}{2}+\delta})$ in L\'evy-Prokhorov distance for $\gamma\in(0,\frac{1}{2})$, but we obtain the Wasserstein convergence rate $O(n^{-\frac{1}{4}+\frac{\gamma}{4(1-\gamma)}+\delta})$ only for $\gamma\in(0,\frac{1}{4})$. See Example \ref{PMM} for details.

As a nontrivial application we consider the deterministic homogenization in fast-slow dynamical systems. In \cite{MR3064670}, Gottwald and Melbourne proved that the slow variable with suitable scales converges weakly to the solution of a stochastic differential equation. Then Antoniou and Melbourne \cite{MR3975864} studied the weak convergence rate of the above problem based on the convergence rate in the WIP with respect to L\'{e}vy-Prokhorov distance. In this paper, we obtain the Wasserstein convergence rate for the homogenization problem based on our results. Although the convergence in Wasserstein distance implies that in L\'{e}vy-Prokhorov distance, the result in \cite{MR3975864} is not contained in our results. Indeed, for uniformly hyperbolic fast systems, we obtain the Wasserstein convergence rate $O(\eps^{\frac{1}{4}-\delta})$, while the weak convergence rate is $O(\eps^{\frac{1}{3}-\delta})$ by \cite{MR3975864} under a little stronger condition, where $\eps$ is identical with $n^{-\frac{1}{2}}$. See Remark
\ref{diss} for details.

The remainder of this paper is organized as follows. In Section 2, we give the definition and basic properties of Wasserstein distances. In Section 3, we review the definitions of nonuniformly expanding maps and nonuniformly hyperbolic diffeomorphisms, and state the main results for discrete systems.
In Section 4, we first introduce the method of martingale approximation and summarize some required properties, then we prove the main results. In Section 5, we consider the convergence rate for flows. In the last section, we give an application to fast-slow systems.

Throughout the paper, we use $1_A$ to denote the indicator function of measurable set $A$. As usual, $a_n=o(b_n)$ means that $\lim_{n\to \infty} a_n/b_n=0$, $a_n=O(b_n)$ means that there exists a constant $C>0$ such that $|a_n|\le C |b_n|$ for all $n\ge 1$, and $\|\cdot\|_{L^p}$ means the $L^p$-norm. For simplicity we write $C$ to denote constants independent of $n$ and $C$ may change from line to line. We use $\rightarrow_{w}$ to denote the weak convergence in the sense of probability measures \cite{MR1700749}. We denote by $C[0,1]$ the space of all continuous functions on $[0,1]$ equipped with the supremum distance $d_C$, that is
\[
d_C(x,y):=\sup_{t\in [0,1]}|x(t)-y(t)|, \quad x,y\in C[0,1].
\]
We use $\P_X$ to denote the law/distribution of random variable $X$ and use $X=_d Y$ to mean $X, Y$ sharing the same distribution.

\section{Preliminaries}

In this section, we review the definition of Wasserstein distances and some important properties about the distance. See e.g.\ \cite{MR2091955,Rachev,MR2459454} for details.

Let $(\mathcal{X}, d)$ be a Polish space, i.e.\ a complete separable metric space, equipped with the Borel $\sigma$-algebra $\mathcal{B}$. For given two probability measures $\mu$ and $\nu$ on $\mathcal{X}$, take two random variables $X$ and $Y$ such that $\hbox{law} (X)=\mu, \hbox{law} (Y)=\nu$. Then the pair $(X,Y)$ is called a {\em coupling} of $\mu$ and $\nu$; the joint distribution of $(X,Y)$ is also called a {\em coupling} of $\mu$ and $\nu$.

\begin{de}\label{was}
Let $q\in[1,\infty)$. Then for any two probability measures $\mu$ and $\nu$ on $\mathcal{X}$, the {\em Wasserstein distance of order $q$} between them is defined by
\begin{align*}
\mathcal{W}_{q}(\mu,\nu) & :={\bigg( \inf_{\pi\in\Pi(\mu,\nu)}\int_\mathcal{X} {d(x,y)}^q\mathrm{d}\pi(x,y)\bigg)}^{1/q}\\
 &= \inf \{ [\mathbf{E} {d(X,Y)}^q]^{1/q}; \hbox{law} (X)=\mu, \hbox{law} (Y)=\nu \},\nonumber
\end{align*}
where $\Pi(\mu,\nu)$ is the set of all couplings of $\mu$ and $\nu$.
\end{de}

\begin{prop}{\rm(see \cite[Lemma 5.2]{MR2091955})}\label{exi}
Given two probability measures $\mu$ and $\nu$ on $\mathcal {X}$, the infimum in Definition \ref{was} can be attained for some coupling $(X, Y)$ of $\mu$ and $\nu$.
\end{prop}

Those couplings achieving the infimum in Proposition \ref{exi} are called {\em optimal couplings} of $\mu$ and $\nu$. Note also that the distance
$\mathcal{W}_{q}(\mu,\nu)$ can be bounded above by the $L^q$ distance of any coupling $(X,Y)$ of $\mu$ and $\nu$.

\begin{prop}{\rm(see \cite[Theorem 5.6]{MR2091955} or \cite[Definition 6.8]{MR2459454})}\label{lid}
$\mathcal{W}_{q}(\mu_n, \mu)\to 0$ if and only if the following two conditions hold:
\begin{enumerate}
  \item $\mu_{n}\rightarrow_{w}\mu;$
  \item $\int_{\mathcal X} d(x,x_{0})^{q}\rmd\mu_{n}(x)\rightarrow \int_{\mathcal X} d(x,x_{0})^{q}\rmd\mu(x)$ for some $($thus any$)$ $x_{0}\in\mathcal{X}$.
\end{enumerate}
 In particular, if $d$ is bounded, then the convergence with respect to $\mathcal{W}_{q}$ are equivalent to the weak convergence.
\end{prop}

\begin{prop}\label{LIP}
Suppose that $\mathcal {G}:\mathcal{X}\rightarrow \mathcal{X}$ is Lipschitz continuous with constant $K$. Then for any two probability measures $\mu$ and $\nu$ on $\mathcal{X}$ and $q\in[1,\infty)$, we have
\[\mathcal{W}_q(\mu\circ \mathcal{G}^{-1},\nu\circ \mathcal{G}^{-1})\le K \mathcal{W}_q(\mu,\nu).\]
\end{prop}
\begin{proof}
By Proposition \ref{exi}, we can choose an optimal coupling $(X, Y)$ of $\mu$ and $\nu$ such that
\[
[\E d(X, Y)^{q}]^{1/q}= \mathcal{W}_q(\mu,\nu).
\]
Then
\begin{align*}
&\mathcal{W}_q(\mu\circ \mathcal{G}^{-1},\nu\circ \mathcal{G}^{-1})\le [\E d(\mathcal{G}(X),\mathcal{G}(Y))^{q}]^{1/q}\\
&\le K[\E d(X,Y)^{q}]^{1/q}= K \mathcal{W}_q(\mu,\nu).
\end{align*}
\end{proof}

\begin{rem}
In the following, we use the notation $\mathcal{W}_p(X,Y)$ to mean $\mathcal{W}_p(\P_X, \P_Y)$ for the sake of simplicity. But we should keep in mind that $(X,Y)$ need not be an optimal coupling of $(\P_X, \P_Y)$.
\end{rem}

The following result is known; see e.g. \cite[Lemma 5.3]{MR2091955} or \cite[Corollary 8.3.1]{Rachev} for details. But the forms or proofs in these references
are different from the following one, which is more appropriate for our purpose. For the convenience of the reader, we also give a proof.

\begin{prop}\label{levy}
For any given probability measures $\mu$ and $\nu$ on $\mathcal{X}$ and $p\in[1,\infty)$, we have
\[
\pi(\mu,\nu)\le \mathcal{W}_{p}(\mu,\nu)^{\frac{p}{p+1}},
\]
where $\pi$ is the L\'{e}vy-Prokhorov distance defined by
\begin{equation*}
\pi(\mu,\nu):=\inf\{\eps> 0: \mu(A)\le \nu(A^{\eps})+\eps \quad \hbox{for ~all ~closed ~sets~} A\in \mathcal{B}\}.
\end{equation*}
Here $A^{\eps}$ denotes the $\eps$-neighborhood of $A$.
\end{prop}
\begin{proof}
Let $A$ be a closed set. Then for any coupling $(X,Y)$ of $\mu$ and $\nu$, we have
\begin{align*}
\P(X\in A)&\le \P(Y\in A^{\eps})+\P(d(X,Y)\ge \eps)\\
&\le \P(Y\in A^{\eps})+\frac{\E d(X,Y)^p}{\eps^p}.
\end{align*}
Note that $\P(X\in A)$ and $\P(Y\in A^{\eps})$ depend on $X$ and $Y$ only through their distributions. So by the arbitrariness of the coupling $(X,Y)$ of $\mu$ and $\nu$,
\[
\P(X\in A)\le \P(Y\in A^{\eps})+ \eps^{-p}\mathcal{W}_{p}(\mu,\nu)^{p}.
\]
Choosing $\eps=\mathcal{W}_{p}(\mu,\nu)^{\frac{p}{p+1}}$, we deduce that $\P(X\in A)\le \P(Y\in A^{\eps}) + \eps$. Hence
\[
\pi(\mu,\nu)\le \mathcal{W}_{p}(\mu,\nu)^{\frac{p}{p+1}}.
\]
\end{proof}

\section{Nonuniformly expanding/hyperbolic maps}
\setcounter{equation}{0}

\subsection{Nonuniformly expanding map}
Let $(M,d)$ be a bounded metric space with Borel probability measure $\rho$. Let $T:M\rightarrow M$ be a non-singular (i.e. $\rho(T^{-1}E)=0$ if and only if $\rho(E)=0$ for all Borel measurable sets $E$), ergodic transformation. Suppose that $Y$ is a subset of $M$ with positive measure, and $\{Y_j\}$ is an at most countable measurable partition of $Y$ with $\rho(Y_j)>0$. Let $R:Y\rightarrow \mathbb{Z}^{+}$ be an integrable function which is constant on each $Y_j$ and $T^{R(y)}(y)\in Y$ for all $y\in Y$. We call $R$ the {\em return time} and $F=T^{R}:Y\rightarrow Y$ is the corresponding {\em induced map}. We do not require that $R$ is the first return time to $Y$.

Let $\nu=\frac{\mathrm{d}\rho|_Y}{\mathrm{d}\rho|_Y\circ F}$ be the inverse Jacobian of $F$ with respect to $\rho$. We assume that there are constants $\lambda>1$, $K, C>0$ and $\eta\in (0,1]$ such that for any $x,y$ in a same partition element $Y_j$,
\begin{enumerate}
  \item $F|_{Y_j}=T^{R(Y_j)}:Y_j\rightarrow Y$ is a (measure-theoretic) bijection for each $j$,
  \item $d(Fx,Fy)\geq\lambda d(x,y)$,
  \item $d(T^{l}x,T^{l}y)\leq C d(Fx,Fy)$ for all $0\leq l < R(Y_j)$,
  \item $|\log\nu(x)-\log\nu(y)|\leq Kd(Fx,Fy)^{\eta}$.
\end{enumerate}
Then such a dynamical system $T:M\rightarrow M$ is  a {\em nonuniformly expanding map}. If $R\in L^p(Y)$ for some $p\ge1$, then we call $T:M\rightarrow M$ a nonuniformly expanding map of {\em order $p$}. It is standard that there is a unique absolutely continuous $F$-invariant probability measure $\mu_Y$ on $Y$ with respect to the measure $\rho$.

We define the {\em Young tower} as in \cite{MR1637655,MR1750438}. Let $\Delta:=\{(x, l):x\in Y, l=0,1,\ldots,R(x)-1\}$, and define an {\em extension map} $f:\Delta\rightarrow\Delta$ by
\[
f(x,l):=
\begin{cases}
(x,l+1), & \text{if } l+1<R(x),\\
(Fx,0), & \text{if }  l+1=R(x).
\end{cases}
\]
We have a {\em projection map} $\pi_\Delta:\Delta\rightarrow M$ given by $\pi_\Delta(x,l):=T^lx$ and it is a semiconjugacy satisfying $T\circ \pi_\Delta=\pi_\Delta\circ f$. Then we obtain an ergodic $f$-invariant probability measure $\mu_\Delta$ on $\Delta$ given by $\mu_\Delta:=\mu_Y\times m/\int_Y R \rmd \mu_Y$, where $m$ denotes the counting measure on $\N$. Hence there exists an {\em extension space} $(\Delta, \mathcal{M}, \mu_\Delta)$, where $\mathcal{M}$ is the underlying $\sigma$-algebra on $(\Delta, \mu_\Delta)$. Further, the push-forward measure $\mu=(\pi_\Delta)_\ast\mu_\Delta$ is an absolutely continuous $T$-invariant probability measure.

Given a H\"{o}lder observable $v:M\rightarrow \R$ with exponent $\eta\in (0,1]$, define
\[
|v|_\infty:=\sup_{x\in M}|v(x)|, \quad |v|_\eta:=\sup_{x\neq y}\frac{|v(x)-v(y)|}{d(x,y)^{\eta}}.
\]
Let $C^\eta(M)$ denote the Banach space of H\"{o}lder observables with norm $\|v\|_\eta=|v|_\infty+|v|_\eta<\infty.$ Consider the following continuous processes $W_{n}$ defined by
\begin{equation}\label{exp}
W_{n}(t):=\frac{1}{\sqrt{n}}\bigg[\sum_{j=0}^{[nt]-1}v\circ T^j+(nt-[nt])v\circ T^{[nt]}\bigg],\quad t\in[0,1],
\end{equation}
where $v\in C^\eta(M)$ with $\int_M v \rmd \mu=0$.  Let $v_n:=\sum_{i=0}^{n-1}v\circ T^i$ denote the Birkhoff sum.\\

The following lemma is a summary of known results; see \cite{MR2027296,MR3795019,MR2175992,MR2995657} for details.
\begin{lem}\label{res}
Suppose that $T:M\rightarrow M$ is a nonuniformly  expanding map of order $p\ge2$. Let $v:M\rightarrow \mathbb{R}$ be a H$\ddot{o}$lder observable with $\int_M v \mathrm{d}\mu=0$. Then the following statements hold.

$(a)$ The limit $\sigma^2=\lim_{n\rightarrow\infty}\int_M(n^{-\frac{1}{2}}v_n)^2\rmd \mu$ exists.

$(b)$ $n^{-\frac{1}{2}}v_{n}\rightarrow_{w} G$ as $n\rightarrow\infty$, where $G$ is normal with mean zero and variance $\sigma^2$.

$(c)$ $W_n\rightarrow_{w} W$ in $C[0,1]$ as $n\rightarrow\infty$, where $W$ is a Brownian motion with mean zero and variance $\sigma^2$.

$(d)$ If $\mu_{Y}(R>n)=O(n^{-(\beta+1)}), \beta>1$, then
\begin{equation*}
\lim_{n\rightarrow\infty}\int_{M}|n^{-\frac{1}{2}}v_{n}|^{q}\mathrm{d}\mu= \E|G|^{q}, \hbox{~for~all~} q\in[0,2\beta).
\end{equation*}

$(e)$ $\left\|\max_{k\leq n}|\sum_{i=0}^{k-1}v\circ T^{i}|\right\|_{L^{2(p-1)}}\leq C\|v\|_{\eta}n^{1/2}$ for all $n\geq1$.
\end{lem}

\begin{proof}
The items $(a)-(c)$ are well known; see e.g. \cite{MR2027296,MR3795019,MR2175992}. The item $(d)$ can be found in \cite[Theorem 3.5]{MR2995657}. For the item $(e)$, see \cite[Corollary 2.10]{MR3795019} for details.
\end{proof}

\begin{rem}\label{opti}
In the case of $(d)$, Melbourne and T\"{o}r\"{o}k \cite{MR2995657} gave examples to illustrate that the $q$-th moments diverge for $q > 2\beta$. Hence the result on the order of convergent moments is essentially optimal.
\end{rem}

\begin{thm}\label{wpc}
Let $T:M\rightarrow M$ be a nonuniformly expanding map of order $p> 2$. Suppose that $v:M\rightarrow \mathbb{R}$ is a H$\ddot{o}$lder observable with $\int_M v \mathrm{d}\mu=0$. Then $\mathcal{W}_{q}(W_{n},W)\to 0$ in $C[0,1]$ for all $1\le q< 2(p-1)$.
\end{thm}

\begin{proof}
It follows from Lemma \ref{res}$(e)$ that $W_n$ has a finite moment of order $2(p-1)$. This together with the fact $W_n\to_{w} W$ as $n\to\infty$ in Lemma \ref{res}$(c)$ implies that for each $q<2(p-1)$
\[
\lim_{n \to \infty}\mathbf{E}\sup_{t\in [0,1]}|W_n(t)|^q=\mathbf{E}\sup_{t\in [0,1]}|W(t)|^q
\]
by \cite[Theorem 4.5.2]{Chung}. On the other hand, by the fact $W_n: M\to C[0,1]$ and the definition of push-forward measures we have
\[
\int_{C[0,1]} d_C(x,0)^q \ \rmd \mu\circ W_n^{-1}(x)=\int_M  \sup_{t\in [0,1]}|W_n(t,\omega)|^q \rmd \mu(\omega)=\mathbf{E}\sup_{t\in [0,1]}|W_n(t)|^q;
\]
similarly,
\[
\int_{C[0,1]} d_C(x,0)^q \ \rmd \mu\circ W^{-1}(x)=\mathbf{E}\sup_{t\in [0,1]}|W(t)|^q.
\]
Hence
\[
\lim_{n \to \infty}\int_{C[0,1]}d_C(x,0)^q \ \rmd \mu\circ W_n^{-1}(x)=\int_{C[0,1]}d_C(x,0)^q \ \rmd \mu\circ W^{-1}(x).
\]
By taking $\mu_n=\mu\circ W_n^{-1}, \mu=\mu\circ W^{-1}$ and $x_0=0$ in Proposition \ref{lid} and the fact $W_n\to_{w} W$ in Lemma \ref{res}$(c)$, the result follows.
\end{proof}

\begin{thm}\label{thnon}
Let $T:M\rightarrow M$ be a nonuniformly expanding map of order $p\ge 4$ and suppose that $v:M\rightarrow \mathbb{R}$ is a H$\ddot{o}$lder observable with $\int_M v \mathrm{d}\mu=0$. Then there exists a constant $C>0$ such that $\mathcal{W}_{\frac{p}{2}}(W_{n},W)\leq Cn^{-\frac{1}{4}+\frac{1}{4(p-1)}}$ for all $n\geq 1$.
\end{thm}

We postpone the proof of Theorem \ref{thnon} to Section 4.
\begin{rem}\label{wqr}
(1) Since $\mathcal{W}_{q}$ $\le \mathcal{W}_{p}$ for $q\le p$, Theorem \ref{thnon} provides an estimate for $\mathcal{W}_{q}(W_{n},W)$ for all $1\le q\le p/2$, $p\ge 4$.

(2) Our result implies a convergence rate $O(n^{-\frac{1}{4}+\delta'})$ with respect to L\'{e}vy-Prokhorov distance, where $\delta'$ depends only on $p$ and $\delta'$ can be arbitrarily small as $p\to \infty$. Indeed, for two given probability measures $\mu$ and $\nu$, we have $\pi(\mu,\nu)\le \mathcal{W}_{p}(\mu,\nu)^{\frac{p}{p+1}}$; see Proposition \ref{levy}.

(3) The convergence rate in Theorem \ref{thnon} may not be optimal. However, it is well known that one cannot get a better result than $O(n^{-\frac{1}{4}})$ by means of the Skorokhod embedding theorem; see \cite{MR0324738,MR301782} for details.
\end{rem}

\begin{exam}[Pomeau-Manneville intermittent maps]\label{PMM}
A typical example of nonuniformly expanding systems with polynomial tails is the Pomeau-Manneville intermittent map \cite{MR1695915,MR576270}. Consider the map $T:[0,1]\rightarrow[0,1]$ given by
\begin{align}\label{P-M}
  T(x)=
  \begin{cases}
  x(1+2^\gamma x^\gamma) & x\in[0,\frac{1}{2}),\\
  2x-1  & x\in[\frac{1}{2},1],
  \end{cases}
\end{align}
where $\gamma\ge 0$ is a parameter. When $\gamma=0$, this is $Tx=2x $ mod $1$ which is a uniformly expanding system.
It is well known that for each $0\le \gamma<1$, there is a unique absolutely continuous invariant probability measure $\mu$.
By \cite{MR1750438}, for $0<\gamma<1$, the map can be modelled by a Young tower with tails $O(n^{-\frac{1}{\gamma}})$. Further for $\gamma\in[0,\frac{1}{2})$, the CLT and WIP hold for H\"{o}lder continuous observables.
We restrict the parameter $\gamma\in(0,\frac{1}{2})$, then the map is a nonuniformly expanding system of order $p$ for any $p<\frac{1}{\gamma}$.
By Theorem \ref{thnon} we obtain $\mathcal{W}_{\frac{p}{2}}(W_{n},W)\leq Cn^{-\frac{1}{4}+\frac{\gamma}{4(1-\gamma)}+\delta}$ for all $\gamma\in(0,\frac{1}{4})$.
\end{exam}

\begin{exam}[Viana maps]
Consider the Viana maps \cite{MR1471866} $T_\alpha:S^1\times\R\rightarrow S^1\times\R$
\[
T_\alpha(\omega, x)=(l\omega \text{~mod~} 1, a_0+\alpha\sin 2\pi \omega-x^2).
\]
Here $a_0\in(1,2)$ is chosen in such a way that $x=0$ is a preperiodic point for the map $g(x)=a_0-x^2$, and $\alpha$ is fixed sufficiently small, $l\in \N$ with $l\ge 16$. The results in \cite{MR2233699} show that any $T$ close to the map $T_\alpha$ in the $C^3$ topology, can be modelled by a Young tower with stretched exponential tails, which is a nonuniformly expanding map of order $p$ for all $p\ge 1$. Hence by Theorem \ref{thnon}, for all $p\geq 4$, $\mathcal{W}_{\frac{p}{2}}(W_{n},W)\leq Cn^{-\frac{1}{4}+\frac{1}{4(p-1)}}$.
\end{exam}

\subsection{Nonuniformly hyperbolic diffeomorphism}
In this subsection, we introduce the main results for nonuniformly hyperbolic systems in the sense of Young \cite{MR1637655, MR1750438}. In this case, we follow the argument in \cite{MR3795019,MR2175992}.

Let $T:M \rightarrow M$ be a diffeomorphism (possibly with singularities\footnote{The meaning of singularity here is in the sense of Young \cite{MR1637655}, different from that of nonuniformly expanding maps at the beginning of Section 3.1.}) defined on a Riemannian manifold $(M,d)$. As in \cite{MR2175992}, consider a subset $Y\subset M$ which has a hyperbolic product structure; that is, there exist a continuous family of unstable disks $\{W^u\}$ and a continuous family of stable disks $\{W^s\}$ such that

(1) $dim W^s+dim W^u=dim M$,

(2) each $W^u$-disk is transversal to each $W^s$-disk in a single point,

(3) $Y=(\cup W^u)\cap(\cup W^s)$.\\
For $x\in Y$, $W^s(x)$ denotes the element in $\{W^s\}$ containing $x$.

Furthermore, there is a measurable partition $\{Y_j\}$ of $Y$ such that each $Y_j$ is a union of elements in $\{W^s\}$ and a $W^u$ such that each element of $\{W^s\}$ intersects $W^u$ in one point. Defining an integrable return time $R:Y\rightarrow\mathbb{Z}^{+}$ which is constant on each partition $Y_j$, we can get the corresponding induced map $F=T^{R}:Y\rightarrow Y$. The {\em separation time} $s(x,y)$ is the greatest integer $n\ge 0$ such that $F^nx,F^ny$ lie in the same partition element of $Y$.

We assume that there exist $C>0$ and $\gamma\in(0,1)$ such that

(1) $F(W^s(x))\subset W^s(Fx)$ for all $x\in Y$;

(2) $d(T^n(x),T^n(y))\le C \gamma^n$ for all $x\in Y$, $y\in W^s(x)$ and $n\ge 0$;

(3) $d(T^n(x),T^n(y))\le C \gamma^{s(x,y)}$ for $x,y\in W^u$ and $0\le n<R$.

As for the nonuniformly expanding map, we can define a Young tower.
Let $\Delta:=\{(x,l):x\in Y, l=0,1,\ldots,R(x)-1\}$ and define an extension map $f:\Delta\rightarrow\Delta$
\[
f(x,l):=
\begin{cases}
(x,l+1) & \text{if } l+1<R(x),\\
(Fx,0) & \text{if } l+1=R(x).
\end{cases}
\]
We have a projection map $\pi_\Delta:\Delta\rightarrow M$ given by $\pi_\Delta(x,l):=T^lx$ and it is a semiconjugacy satisfying $T\circ \pi_\Delta=\pi_\Delta\circ f$.

Let $\bar{Y}=Y/\thicksim$, where $y\thicksim y'$ if $y'\in W^s(y)$; denote by $\bar{\pi}:Y\rightarrow\bar{Y}$ the natural projection. We can also obtain a partition $\{\bar{Y}_j\}$ of $\bar{Y}$, a well-defined return time $\bar{R}:\bar{Y}\rightarrow \mathbb{Z}^{+}$ and a corresponding induced map $\bar{F}:\bar{Y}\rightarrow\bar{Y}$  as in the case of $Y$. In addition, we assume that

(1)$\bar{F}|_{\bar{Y}_j}=\bar{T}^{\bar{R}(\bar{Y}_j)}:\bar{Y}_j\rightarrow\bar{Y}$ is a bijection for each $j$;

(2)$\nu_0=\frac{d\bar{\rho}}{d\bar{\rho}\circ \bar{F}}$ satisfies $|\log\nu_0(y)-\log\nu_0(y')|\le K\gamma^{s(y,y')} $, for all $y,y'\in \bar{Y}_j$, where $\bar{\rho}=\bar{\pi}_\ast\rho$ with $\rho$ being the Riemannian measure.

Let $\bar{f}:\bar{\Delta}\rightarrow\bar{\Delta}$ denote the corrosponding extension map. The projection $\bar{\pi}:Y\rightarrow\bar{Y}$ extends to the projection $\bar{\pi}:\Delta\rightarrow\bar{\Delta}$; here we use the same notation $\bar{\pi}$ which should not cause confusion. There exist an $\bar{f}$-invariant probability measure $\bar{\mu}$ on $\bar{\Delta}$ and an $f$-invariant probability measure $\mu_{\Delta}$ on $\Delta$, such that $\bar{\pi}:\Delta\rightarrow\bar{\Delta}$ and $\pi_\Delta:\Delta\rightarrow M$ are measure-preserving.

\begin{thm}
Let $T:M\rightarrow M$ be a nonuniformly hyperbolic transformation of order $p> 2$. Suppose that $v:M\rightarrow \mathbb{R}$ is a H$\ddot{o}$lder observable with $\int_M v \mathrm{d}\mu=0$. Then $\mathcal{W}_{q}(W_{n},W)\to 0$ in $C[0,1]$ for all $1\le q< 2(p-1)$.
\end{thm}

\begin{proof}
By \cite[Corollary 5.5]{MR3795019}, $W_n$ has a finite moment of order $2(p-1)$. The remaining proof is similar to Theorem \ref{wpc}.
\end{proof}

\begin{thm}\label{hyp}
Let $T:M\rightarrow M$ be a nonuniformly hyperbolic transformation of order $p\ge 4$ and suppose that $v:M\rightarrow \mathbb{R}$ is a H$\ddot{o}$lder observable with $\int_M v \mathrm{d}\mu=0$. Then there exists a constant $C>0$ such that $\mathcal{W}_{\frac{p}{2}}(W_{n},W)\leq Cn^{-\frac{1}{4}+\frac{1}{4(p-1)}}$ for all $n\geq 1$.
\end{thm}
We postpone the proof of Theorem \ref{hyp} to the next section.


\begin{exam}[Nonuniformly expanding/hyperbolic systems with exponential tails]
In this case, the return time $R\in L^p$ for all $p$. Hence for all $p\geq 4$, $\mathcal{W}_{\frac{p}{2}}(W_{n},W)\leq Cn^{-\frac{1}{4}+\frac{1}{4(p-1)}}$. Specific examples are the following.
\begin{itemize}
 \item Some partially hyperbolic systems with a mostly contracting direction \cite{MR1919371, MR1782146}.
 \item Unimodal maps and multimodal maps as in \cite{MR3640023} for a fixed system.
 \item H\'{e}non-type attractors \cite{MR422932}. Let $T_{a,b}:\R^2\rightarrow\R^2$ be defined by $T_{a,b}(x,y)=(1-ax^2+y, bx)$ for $a<2$, $b>0$, where $b$ is small enough depending on $a$. It follows from \cite{MR1218323, MR1755436} that $T$ admits an SRB measure and $T$ can be modelled by a Young tower with exponential tails.
\end{itemize}
\end{exam}

\section{Proof of Theorem \ref{thnon} and Theorem \ref{hyp}}
\setcounter{equation}{0}

\subsection{Martingale Approximation}
The martingale approximation method \cite{MR0251785} is one of the main methods for studying statistical limit properties. In \cite{MR3795019}, Korepanov et al obtained a new version of martingale-coboundary decomposition, which is applicable to nonuniformly hyperbolic systems. In this subsection, we recall some required properties in \cite{MR3795019}.

\begin{prop}\label{dec}
Let $T:M\rightarrow M$ be a nonuniformly expanding map of order $p\ge 1$ and suppose that $v:M\rightarrow \mathbb{R}$ is a H$\ddot{o}$lder observable with $\int_M v \mathrm{d}\mu=0$. Then there is an extension $f:\Delta\rightarrow\Delta$ of $T$ such that for any $v\in C^\eta(M)$ there exist $m\in L^{p}(\Delta)$ and $\chi\in L^{p-1}(\Delta)$ with
\[
v\circ\pi_\Delta=m+\chi\circ f-\chi,\qquad \mathbf{E}(m|f^{-1}\mathcal{M})=0.
\]
Moreover, there is a constant $C>0$ such that for all $v\in C^\eta(M)$,
\[
\|m\|_{L^p}\le C\|v\|_{\eta},\qquad \|\chi\|_{L^{p-1}}\le C\|v\|_{\eta},
\]
and for $n\geq 1$,
\[
\left\|\max_{0\le j\le n}|\chi\circ f^j-\chi|\right\|_{L^p}\le C\|v\|_{\eta}n^{1/p}.
\]
\end{prop}

\begin{proof}
The proposition is a summary of Propositions 2.4, 2.5 and 2.7 in \cite{MR3795019}.
\end{proof}

\begin{prop}
Fix $n\ge 1$. Then $\{m\circ f^{n-i},f^{-(n-i)}\mathcal{M};1\le i\le n\}$ is a martingale difference sequence.
\end{prop}

\begin{proof}
See for example \cite[Proposition 2.9]{MR3795019}.
\end{proof}

\begin{prop}\label{mom}
If $p\geq2$, then $\left\|\max_{k\leq n}|\sum_{i=1}^{k}m\circ f^{n-i}|\right\|_{L^{p}}\leq C\|m\|_{L^p}n^{1/2}$ for all $n\ge 1$.
\end{prop}

\begin{proof}
See the proof in \cite[Corollary 2.10]{MR3795019}.
\end{proof}

\subsection{Proof of Theorem \ref{thnon}}

Define
\[
\zeta_{n,j}:=\frac{1}{\sqrt{n}\sigma}m\circ f^{n-j},\ \mathcal{F}_{n,j}:=f^{-(n-j)}\mathcal{M}, \quad\hbox{for } 1\le j\le n.
\]
For $1\le l\le n$, define the conditional variance
\[
V_{n,l}:=\sum_{j=1}^{l}\mathbf{E}(\zeta_{n,j}^2|\mathcal{F}_{n,j-1}).
\]
And we set $V_{n,0}=0$.

Define the stochastic process $X_n$ with sample paths in $C[0,1]$ by
\begin{eqnarray}\label{xn}
X_{n}(t):=\sum_{j=1}^{k}\zeta_{n,j}+\frac{tV_{n,n}-V_{n,k}}{V_{n,k+1}-V_{n,k}}\zeta_{n,k+1}, \qquad \textrm{if}~~~ V_{n,k}\leq tV_{n,n}<V_{n,k+1}.
\end{eqnarray}

\medskip

{\bf Step 1. Estimate of the Wasserstein distance between $X_n$ and $B$.} Let $B$ be a standard Brownian motion, i.e. $B=_d\frac1\sigma W$.

\begin{lem}\label{xnw}
Let $p\ge 4$. Then for any $\delta>0$ there exists a constant $C>0$ such that $\mathcal{W}_{\frac{p}{2}}(X_{n},B)\leq C n^{-(\frac{1}{4}-\delta)}$ for all $n\geq 1$.
\end{lem}

\begin{proof}
1. Fix $n>0$. It suffices to deal with a single row of the array $\{\zeta_{n,j},\mathcal{F}_{n,j}, 1\le j\le n\}$. By the Skorokhod embedding theorem (see Theorem \ref{ske}), there exists a probability space (depending on $n$) supporting a standard Brownian motion, still denoted by $B$ which should not cause confusion, and a sequence of nonnegative random variables $\tau_1,\ldots, \tau_n$ such that for $T_i=\sum_{j=1}^{i}\tau_j$ we have $\sum_{j=1}^{i}\zeta_{n,j}=B(T_i)$ with $1\le i\le n$. Specially, we set $T_{0}=0$. Then on this probability space and for this Brownian motion, we aim to show that for any $\delta>0$ there exists a constant $C>0$ such that
\begin{align*}
\bigg\|\sup_{t\in[0,1]}|X_{n}(t)-B(t)|\bigg\|_{L^{\frac{p}{2}}}\leq Cn^{-(\frac{1}{4}-\delta)} \quad \hbox{for~all~}n\ge1.
\end{align*}
Thus the result follows from Definition \ref{was}.

For ease of exposition when there is no ambiguity, we will write $\zeta_j$ and $V_k$ instead of $\zeta_{n,j}$ and $V_{n,k}$ respectively. Then by \eqref{xn} we have
\begin{align}\label{set}
X_{n}(t)=B(T_{k})+\bigg(\frac{tV_{n}-V_{k}}{V_{k+1}-V_{k}}\bigg)\big(B(T_{k+1})-B(T_{k})\big),\quad \hbox{if}~V_{k}\leq tV_{n}<V_{k+1}.
\end{align}

2. Note that Theorem \ref{ske}(3) implies
\[
T_k-V_k=\sum_{i=1}^{k}\big(\tau_{i}-\mathbf{E}(\tau_{i}|\mathcal{B}_{i-1})\big),\quad 1\le k\le n,
\]
where $\mathcal{B}_{i}$ is the $\sigma$-field generated by all events up to $T_i$ for $1\le i\le n$. Therefore $\{T_k-V_k, \mathcal{B}_{k}, 1\le k\le n\}$ is a martingale. By the Burkholder inequality and conditional Jensen inequality, for all $p\ge 4$, we have
\begin{align*}\label{mar}
\bigg\|\max_{1\le k\le n}|T_k-V_{k}|\bigg\|_{L^{\frac{p}{2}}} &\le Cn^{\frac{1}{2}} \max_{1\le k\le n}\big\|\tau_{k}-\mathbf{E}(\tau_{k}|\mathcal{B}_{k-1})\big\|_{L^{\frac{p}{2}}}\\
&\le Cn^{\frac{1}{2}} \max_{1\le k\le n}\big\|\tau_{k}\big\|_{L^{\frac{p}{2}}}.
\end{align*}
It follows from Theorem \ref{ske}(4) that $\mathbf{E}(\tau_{k}^{p/2})\le 2\Gamma(\frac{p}{2}+1)\mathbf{E}(\zeta_{k}^{p})$ for each $k$. So
\begin{equation}\label{TkVk}
\bigg\|\max_{1\le k\le n}|T_k-V_{k}|\bigg\|_{L^{\frac{p}{2}}} \le Cn^{\frac{1}{2}} \max_{1\le k\le n}\big\|\zeta_{k}\big\|_{L^p}^2= Cn^{-\frac{1}{2}}\|m\|_{L^p}^{2}.
\end{equation}
On the other hand, it follows from \cite[Proposition 4.1]{MR3975864} that
\begin{align}\label{Vn1}
\|V_{n}-1\|_{L^{\frac{p}{2}}}\le  C n^{-\frac{1}{2}}\|v\|_{\eta}^{2}.
\end{align}

3. Based on the above estimates, by Chebyshev's inequality we have
\begin{equation}\label{eqq}
\begin{split}
\mu(|T_{n}-1|>1)&\le \mathbf{E}|T_{n}-1|^{\frac{p}{2}} \le 2^{\frac{p}{2}-1}\left\{\mathbf{E}|T_{n}-V_{n}|^{\frac{p}{2}}+\mathbf{E}|V_{n}-1|^{\frac{p}{2}}\right\}\\
&\le Cn^{-\frac{p}{4}}(\|m\|_{L^p}^{p}+\|v\|_{\eta}^{p}).\\
\end{split}
\end{equation}
According to the H\"{o}lder inequality, \eqref{eqq} and Proposition \ref{mom}, we deduce that
\begin{align*}
I:=&\bigg\| 1_{\{|T_{n}-1|>1\}}\sup_{t\in[0,1]}|X_{n}(t)-B(t)|\bigg\|_{L^{\frac{p}{2}}}\\
\le &\big(\mu(|T_{n}-1|>1)\big)^{1/p}\bigg\| \sup_{t\in[0,1]}|X_{n}(t)-B(t)|\bigg\|_{L^{p}}\\
\le &\big(\mu(|T_{n}-1|>1)\big)^{1/p}\bigg(\bigg\|\sup_{t\in[0,1]}|X_{n}(t)|\bigg\|_{L^{p}}+\bigg\|\sup_{t\in[0,1]}|B(t)|\bigg\|_{L^{p}}\bigg)\\
\le &Cn^{-\frac{1}{4}}.
\end{align*}

4. We now estimate $|X_{n}-B|$ on the set $\{|T_{n}-1|\le 1\}$:
\begin{align*}
&\bigg\| 1_{\{|T_{n}-1|\le 1\}}\sup_{t\in[0,1]}|X_{n}(t)-B(t)|\bigg\|_{L^{\frac{p}{2}}}\\
\le &\bigg\| 1_{\{|T_{n}-1|\le 1\}}\sup_{t\in[0,1]}|X_{n}(t)-B(T_{k})|\bigg\|_{L^{\frac{p}{2}}}+\bigg\|1_{\{|T_{n}-1|\le 1\}}\sup_{t\in[0,1]}|B(T_{k})-B(t)|\bigg\|_{L^{\frac{p}{2}}}\\
 =: & I_1 + I_2.
\end{align*}
For $I_1$, it follows from \eqref{set} that
\[
\sup_{t\in[0,1]}|X_{n}(t)-B(T_{k})|\le \max_{0\le k\le n-1}|B(T_{k+1})-B(T_{k})|=\max_{0\le k\le n-1}|\zeta_{k+1}|.
\]
By Theorem \ref{mpe},
\begin{align*}\label{zeta}
I_1=&\bigg\| 1_{\{|T_{n}-1|\le 1\}}\sup_{t\in[0,1]}|X_{n}(t)-B(T_{k})|\bigg\|_{L^{p}}\nonumber\\
\le &\bigg\| 1_{\{|T_{n}-1|\le 1\}}\max_{0\le k\le n-1}|\zeta_{k+1}|\bigg\|_{L^{p}}\nonumber\\
\le &\bigg\| \max_{0\le k\le n-1}|\zeta_{k+1}|\bigg\|_{L^{p}}\nonumber\\
\le &Cn^{-\left(\frac{1}{2}-\frac{1}{p}\right)}.
\end{align*}

5. We now consider $I_{2}$ on the set $\{|T_{n}-1|\le 1\}$. Take $p_1>p$, then it is well known that
\begin{equation}\label{bts}
\mathbf{E}|B(t)-B(s)|^{p_1}\le c|t-s|^{\frac{p_1}{2}}, \quad \text{for~ all~} s,t\in [0,2].
\end{equation}
So it follows from Kolmogorov's continuity theorem that for each $0<\gamma<\frac{1}{2}-\frac{1}{p_1}$, the process $B(\cdot)$ admits a version, still denoted by $B$, such that for almost all $\omega$ the sample path $t\mapsto B(t,\omega)$ is H\"{o}lder continuous with exponent $\gamma$ and
\begin{equation*}
\bigg\|\sup_{s,t\in[0,2]\atop s\neq t}\frac{|B(s)-B(t)|}{|s-t|^{\gamma}}\bigg\|_{L^{p_1}}< \infty.
\end{equation*}
In particular,
\begin{equation}\label{holder}
\bigg\|\sup_{s,t\in[0,2]\atop s\neq t}\frac{|B(s)-B(t)|}{|s-t|^{\gamma}}\bigg\|_{L^{p}}< \infty.
\end{equation}

As for $|T_{k}-t|$, we have
\begin{align*}
\sup_{t\in[0,1]}|T_{k}-t|\le &\max_{0\le k\le n-1}\sup_{t\in[\frac{V_{k}}{V_{n}},\frac{V_{k+1}}{V_{n}})}|T_{k}-t|\\
\le &\max_{0\le k\le n-1}\bigg|T_{k}-\frac{V_{k}}{V_{n}}\bigg|+\max_{0\le k\le n-1}\sup_{t\in[\frac{V_{k}}{V_{n}},\frac{V_{k+1}}{V_{n}})}\bigg|\frac{V_{k}}{V_{n}}-t\bigg|\\
\le &\max_{0\le k\le n}\bigg|T_{k}-\frac{V_{k}}{V_{n}}\bigg|+\max_{0\le k\le n-1}\left|\frac{V_{k+1}}{V_{n}}-\frac{V_{k}}{V_{n}}\right|\\
\le &\max_{0\le k\le n}\big|T_{k}-V_{k}\big|+\max_{0\le k\le n}\bigg|V_{k}-\frac{V_{k}}{V_{n}}\bigg|+\max_{0\le k\le n-1}\left|\frac{V_{k+1}}{V_{n}}-V_{k+1}\right|\\
&\quad +\max_{0\le k\le n-1}\left|V_{k+1}-V_{k}\right|+\max_{0\le k\le n-1}\left|V_{k}-\frac{V_{k}}{V_{n}}\right|\\
\le &\max_{0\le k\le n}\big|T_{k}-V_{k}\big| + 3\max_{0\le k\le n}\left|V_{k}-\frac{V_{k}}{V_{n}}\right| +\max_{0\le k\le n-1}\big|V_{k+1}-V_{k}\big|.
\end{align*}
Note that $T_{0}=V_{0}=0$ and $\gamma\le1$, so
\[
\sup_{t\in[0,1]}|T_{k}-t|^{\gamma}\le \max_{1\le k\le n}\left|T_{k}-V_{k}\right|^{\gamma} + 3^\gamma\max_{1\le k\le n}\left|V_{k}-\frac{V_{k}}{V_{n}}\right|^{\gamma} +\max_{0\le k\le n-1}\left|V_{k+1}-V_{k}\right|^{\gamma}.
\]
Hence we have
\begin{align}
&\bigg\|\sup_{t\in[0,1]}|T_{k}-t|^{\gamma}\bigg\|_{L^{p}}\nonumber\\
\le & \bigg\|\max_{1\le k\le n}\big|T_{k}-V_{k}\big|\bigg\|_{L^{\gamma p}}^{\gamma} +3^\gamma \bigg\|\max_{1\le k\le n}\big|V_{k}-\frac{V_{k}}{V_{n}}\big|\bigg\|_ {L^{\gamma p}}^{\gamma}
 +\bigg\|\max_{0\le k\le n-1}\big|V_{k+1}-V_{k}\big|\bigg\|_ {L^{\gamma p}}^{\gamma}.
\end{align}
For the first term, since $\gamma< \frac{1}{2}$, it follows from \eqref{TkVk} that
\begin{equation}\label{Tgam}
\bigg\|\max_{1\le k\le n}|T_{k}-V_{k}|\bigg\|_{L^{\gamma p}}^{\gamma}\le Cn^{-\frac{\gamma}{2}}.
\end{equation}
For the second term, since $|V_{k}-\frac{V_{k}}{V_{n}}|=V_k|1-\frac1{V_n}|$, we have
\[\max_{1\le k\le n}\bigg|V_{k}-\frac{V_{k}}{V_{n}}\bigg|=V_n \bigg|1-\frac1{V_n}\bigg|=|V_{n}-1|.\]
Hence by \eqref{Vn1},
\begin{align}\label{Vgam}
\bigg\|\max_{1\le k\le n}\big|V_{k}-\frac{V_{k}}{V_{n}}\big|\bigg\|_ {L^{\gamma p}}^{\gamma}
=\|V_{n}-1\|_{L^{\gamma p}}^{\gamma}\le Cn^{-\frac{\gamma}{2}}.
\end{align}
As for the last term, note that $|V_{k}-V_{k-1}|=\mathbf{E}(\zeta_{k}^2|\mathcal{F}_{k-1})
=\mathbf{E}\big(\frac{1}{n\sigma^2}m^2|f^{-1}\mathcal{M}\big)\circ f^{n-k}$ for all $1\le k\le n$.
So we have
\begin{align}\label{mgam}
\bigg\|\max_{0\le k\le n-1}\big|V_{k+1}-V_{k}\big|\bigg\|_ {L^{\gamma p}}^{\gamma}=\bigg\|\max_{1\le k\le n}\big|\mathbf{E}\big(\frac{m^2}{n\sigma^2}|f^{-1}\mathcal{M}\big)\circ f^{n-k}\big|\bigg\|_{L^{\gamma p}}^{\gamma}
\le Cn^{-\left(\gamma-\frac{2\gamma}{p}\right)},
\end{align}
where the inequality follows from Theorem \ref{mpe}.

Based on the above estimates \eqref{Tgam}--\eqref{mgam}, we have
\begin{align}\label{te}
\bigg\|\sup_{t\in[0,1]}|T_{k}-t|^{\gamma}\bigg\|_{L^{p}}
\le C\bigg(n^{-\frac{\gamma}{2}}+n^{-\left(\gamma-\frac{2\gamma}{p}\right)}\bigg)
\le C n^{-\frac{\gamma}{2}},
\end{align}
where the last inequality holds since
$\gamma< \frac{1}{2}$, $1-\frac{2}{p}\ge\frac{1}{2}$.

On the set $\{|T_{n}-1|\le 1\}$, note that
\[
\sup_{t\in[0,1]}|B(T_{k})-B(t)|\le \bigg[\sup_{s,t\in[0,2]\atop s\neq t}\frac{|B(s)-B(t)|}{|s-t|^{\gamma}}\bigg]\bigg[\sup_{t\in[0,1]}|T_{k}-t|^{\gamma}\bigg].
\]
Since $0<\gamma<\frac{1}{2}-\frac{1}{p_1}$, by the H\"{o}lder inequality and \eqref{holder}, \eqref{te}
\begin{align*}
I_2=&\bigg\|1_{\{|T_{n}-1|\le 1\}}\sup_{t\in[0,1]}|B(T_{k})-B(t)|\bigg\|_{L^{\frac{p}{2}}}\\
\le &\bigg\|\bigg[\sup_{s,t\in[0,2]\atop s\neq t}\frac{|B(s)-B(t)|}{|s-t|^{\gamma}}\bigg]\bigg[\sup_{t\in[0,1]}|T_{k}-t|^{\gamma}\bigg]\bigg\|_{L^{\frac{p}{2}}}\\
\le &\bigg\|\sup_{s,t\in[0,2]\atop s\neq t}\frac{|B(s)-B(t)|}{|s-t|^{\gamma}}\bigg\|_{L^{p}}\bigg\|\sup_{t\in[0,1]}|T_{k}-t|^{\gamma}\bigg\|_{L^{p}}\\
\le & C n^{-\frac{\gamma}{2}}.
\end{align*}
Note that $p_1$ can be taken arbitrarily large in \eqref{bts}, which implies that $\gamma$ can be chosen sufficiently close to $\frac{1}{2}$. So for any $\delta>0$, we can choose $p_1$ large enough such that $I_2\le Cn^{-\frac{1}{4}+\delta}$. The result now follows from the above estimates for $I,I_1$ and $I_2$.
\end{proof}

{\bf Step 2. Estimate of the convergence rate between $W_n$ and $X_n$.} The proof is almost identical to that in \cite[Section 4.1]{MR3975864}, so we only sketch the proof.

\begin{prop}\cite[Proposition 4.6]{MR3975864}\label{est}
For $n\ge 1$, define $Z_n:=\max_{0\le i,l\le\sqrt{n}}\left|\sum_{j=i\sqrt{n}}^{i\sqrt{n}+l-1}v\circ T^j\right|.$
Then

$(a)$ $|\sum_{j=a}^{b-1}v\circ T^j|\le Z_n((b-a)(n^{\frac{1}{2}}-1)^{-1}+3)$ for all $0\le a<b \le n$.

$(b)$ $\|Z_n\|_{L^{2(p-1)}}\le C\|v\|_{\eta}n^{\frac{1}{4}+\frac{1}{4(p-1)}}$ for all $n\ge 1$.
\end{prop}

\vskip3mm

Define a continuous transformation $g:C[0,1]\rightarrow C[0,1]$ by $g(u)(t):=u(1)-u(1-t)$.

\begin{lem}\label{wnxn}
Let $p>2$. Then there exists a constant $C>0$ such that $\mathcal{W}_{p-1}(g\circ W_{n}\circ \pi_\Delta,\sigma X_{n})\leq Cn^{-\frac{1}{4}+\frac{1}{4(p-1)}}$ for all $n\geq 1$, recalling that $\pi_\Delta:\Delta\to M$
is the projection map.
\end{lem}

\begin{proof}
Since $\mathcal{W}_{p-1}(g\circ W_{n}\circ \pi_\Delta,\sigma X_{n})\le \big\|\sup_{t\in[0,1]}|g\circ W_{n}(t)\circ \pi_\Delta-\sigma X_n(t)|\big\|_{L^{p-1}}$, following the proof of \cite[Lemma 4.7]{MR3975864}, we can obtain the conclusion.
\end{proof}

\begin{proof}[Proof of Theorem \ref{thnon}]
Note that $g\circ g= Id$ and $g$ is Lipschitz with Lip$g$ $\le 2$. It follows from Proposition \ref{LIP} that
\[
\mathcal{W}_{\frac{p}{2}}(W_{n},W)= \mathcal{W}_{\frac{p}{2}}(g(g\circ W_{n}),g(g\circ W))\le 2\mathcal{W}_{\frac{p}{2}}(g\circ W_{n},g\circ W).
\]
Since $\pi_\Delta$ is a semiconjugacy, $W_n \circ \pi_\Delta=_d W_n$. And $g(W)=_d W=_d\sigma B$. By Lemmas \ref{xnw} and \ref{wnxn}, for $p\ge 4$ we have
\begin{align*}
\mathcal{W}_{\frac{p}{2}}(g\circ W_{n},g\circ W)&=\mathcal{W}_{\frac{p}{2}}(g\circ W_{n}\circ \pi_\Delta,W)\\
&\le \mathcal{W}_{\frac{p}{2}}(g\circ W_{n}\circ \pi_\Delta,\sigma X_n)+\mathcal{W}_{\frac{p}{2}}(\sigma X_n ,\sigma B)\\
&\le Cn^{-\frac{1}{4}+\frac{1}{4(p-1)}}+Cn^{-\frac{1}{4}+\delta}\le Cn^{-\frac{1}{4}+\frac{1}{4(p-1)}},
\end{align*}
where the last inequality holds because $\delta>0$ can be taken arbitrarily small.
\end{proof}

\subsection{Proof of Theorem \ref{hyp}}
The proof is based on the following Lemma \ref{nht} which is presented in detail in \cite[Section 5]{MR3795019}.
\begin{lem}\label{nht}
Let $p\ge 1$, $\eta\in(0,1]$. Suppose that $T:M\rightarrow M$ is a nonuniformly hyperbolic transformation with the return time $R\in L^p$ and $v:M\rightarrow \mathbb{R}$ is a H$\ddot{o}$lder observable. Then

(1) $\bar{f}:\bar{\Delta}\rightarrow\bar{\Delta}$ is a nonuniformly expanding map of order p;

(2) there exists $\theta\in(0,1)$ such that for all $v\in C^\eta(M)$ , there exist $\phi\in C^\theta(\bar{\Delta})$ and $\psi\in L^\infty(\Delta)$ such that $v\circ \pi_\Delta=\phi\circ \bar{\pi}+\psi-\psi\circ f$. Moreover, $|\psi|_\infty\le C\|v\|_\eta$, $\|\phi\|_\theta\le C\|v\|_\eta$.
\end{lem}

\begin{proof}[Proof of Theorem \ref{hyp}]
As the definition of $W_n$ in \eqref{exp}, define $\overline{W}_n(t):=\frac{1}{\sqrt{n}}\sum_{j=0}^{nt-1}\phi\circ \bar{f}^j$ for $t=\frac{j}{n}, 1\le j\le n$, and linearly interpolate to obtain the process $\overline{W}_n\in C[0,1]$, where $\phi$ is from Lemma \ref{nht}. By Lemma \ref{nht}, we have $|\overline{W}_n(t)\circ \bar{\pi}-W_n\circ \pi_\Delta|_\infty\le Cn^{-\frac{1}{2}}|\psi|_\infty$ by simple computations.
Since $\bar{\pi}, \pi_\Delta$ are semiconjugacies, $\mathcal{W}_{\frac{p}{2}}(\overline{W}_n,W_{n})=\mathcal{W}_{\frac{p}{2}}(\overline{W}_n\circ \bar{\pi},W_n\circ\pi_\Delta)\le Cn^{-\frac{1}{2}}$. It follows from Lemma \ref{nht} that $\bar{f}$ is a nonuniformly expanding map of order $p$ and $\phi$ is a H\"{o}lder continuous observable with $\int_{\bar{\Delta}} \phi \mathrm{d}\bar{\mu}=0$. By Theorem \ref{thnon}, for $p\ge 4$, $\mathcal{W}_{\frac{p}{2}}(\overline{W}_{n},W)\le Cn^{-\frac{1}{4}+\frac{1}{4(p-1)}}$. Hence $\mathcal{W}_{\frac{p}{2}}(W_{n},W)\le \mathcal{W}_{\frac{p}{2}}(W_{n},\overline{W}_n)+ \mathcal{W}_{\frac{p}{2}}(\overline{W}_{n},W)\le Cn^{-\frac{1}{4}+\frac{1}{4(p-1)}}$.
\end{proof}

\section{Nonuniformly expanding/hyperbolic flows}
\setcounter{equation}{0}

\subsection{Nonuniformly expanding semiflows}
Suppose that $T:M\to M$ is a nonuniformly expanding map of order $p>2$ with ergodic invariant probability measure $\mu$. Let $r:M\to \R^{+}$ be an integrable roof function with $\bar{r}=\int_M r\rmd \mu$. We assume that $\inf r>0$. Define the suspension $M^r=\{(x,u)\in M\times\R:0\le u\le r(x)\}/\thicksim$, where $(x,r(x))$ is identified with $(Tx,0)$. The suspension semiflow $\phi_t:M^r\to M^r$ is given by $\phi_t(x,u)=(x,u+t)$ computed modulo identifications. We call $\phi_t:M^r\to M^r$ a nonuniformly expanding semiflow. Then $\mu^r=(\mu\times l)/\bar{r}$ is an ergodic invariant probability measure for $\phi_t$, where $l$ denotes the Lebesgue measure on the real line.

Now suppose that $v:M^r\to \R$ is a H\"older observable with $\int_{M^r} v \rmd \mu^r=0$. Here $v:M^r\to \R$ is H\"older if $v$ is bounded and $\sup_{(x,u)\neq(y,u)}|v(x,u)-v(y,u)|/d(x,y)^\eta<\infty$ for $\eta\in(0,1]$. Define the continuous processes $W_n\in C[0,1]$,
\begin{align}
W_n(t):=\frac{1}{\sqrt{n}}\int_0^{nt}v\circ\phi_s\rmd s, \quad t\in[0,1].
\end{align}
Assume moreover that the roof function $r:M\to \R^+$ is H\"{o}lder. It follows from \cite[Corollary 2.12]{MR2175992} that the WIP holds for nonuniformly expanding semiflows; that is, $W_n\to_{w}W$ in $C[0,1]$, where $W$ is a Brownian motion with mean zero and variance $\sigma^2\ge 0$.

In the following two theorems, we give the Wasserstein convergence result for $W_n$ and the convergence rate.

\begin{thm}\label{flow}
Let $\phi_t:M^r\to M^r$ be a nonuniformly expanding semiflow. Suppose that the return time $R\in L^p(Y)$ for $p> 2$ and the roof function $r:M\to \R^+$ is H$\ddot{o}$lder. Let $v:M^r\to \R$ be a H\"older observable with $\int_{M^r} v \rmd \mu^r=0$. Then $\mathcal{W}_{q}(W_{n},W)\to 0$ in $C[0,1]$ for all $1\le q< 2(p-1)$.
\end{thm}

\begin{thm}\label{nef}
Let $\phi_t:M^r\to M^r$ be a nonuniformly expanding semiflow. Suppose that the return time $R\in L^p(Y)$ for $p\ge4$ and the roof function $r:M\to \R^+$ is H$\ddot{o}$lder. Let $v:M^r\to \R$ be a H\"older observable with $\int_{M^r} v \rmd \mu^r=0$. Then there exists a constant $C>0$ such that $\mathcal{W}_{\frac{p}{2}}(W_{n},W)\leq Cn^{-\frac{1}{4}+\frac{1}{4(p-1)}}$ for all $n\geq 1$.
\end{thm}

In the remainder of this subsection, we prove Theorems \ref{flow} and \ref{nef}. We define $\tilde{v}:M\to \R$ by setting $\tilde{v}(x):=\int_{0}^{r(x)}v(x,u)\rmd u$. Then $\tilde{v}:M\to \R$ is a H\"older observable with $\int_{M} \tilde{v} \rmd \mu=0$. Also, we define the continuous processes $\widetilde{W}_n\in C[0,1]$ as follows:
\begin{equation}\label{dis}
\widetilde{W}_{n}(t):=\frac{1}{\sqrt{n}}\bigg[\sum_{j=0}^{[nt]-1}\tilde{v}\circ T^j+(nt-[nt])\tilde{v}\circ T^{[nt]}\bigg],\quad t\in[0,1].
\end{equation}
In the discrete time case, we have obtained the corresponding convergence result; that is,
$\mathcal{W}_{q}(\widetilde W_{n},\widetilde W)\to 0$ in $C[0,1]$ for all $1\le q< 2(p-1)$, where $\widetilde W=_{d}(\bar r)^{1/2}W$ is a Brownian motion. Note that the probability space for $\widetilde{W}_{n}$ is $(M,\mu)$, and we can extend $\widetilde{W}_{n}$ trivially to $(M^r, \mu^r)$ by setting $\widetilde{W}_{n}(x,u)=\widetilde{W}_{n}(x)$; we will regard $\widetilde{W}_{n}$ defined on $(M,\mu)$ or $(M^r, \mu^r)$ alternately
for our purpose, which should not cause confusion. The similar convention applies also to other processes or random variables in what follows.

Denote $v_t=\int_{0}^{t}v\circ \phi_s \rmd s$, $\tilde{v}_n=\sum_{j=0}^{n-1}\tilde{v}\circ T^j$, $r_n=\sum_{j=0}^{n-1}r\circ T^j$. For $(x,u)\in M^r$ and $t>0$, we define the lap number $N(t)=N(x,u,t)\in \N$:
\[
N(t):=\max\{n\ge 0:r_n(x)\le u+t\}.
\]

\begin{prop}\label{ntt}
 $\left\|\sup_{t\in[0,1]}\big|N(nt)-[\frac{nt}{\bar r}]\big|\right\|_{L^{2(p-1)}(M^r)}=O(n^{1/2})$ as $n\to\infty$.
\end{prop}

\begin{proof}
Define $\tilde{r}_{j}:=\frac{1}{\bar r}r_{j}-j$ for $j\in\N$. Then
\[
\tilde{r}_{N(nt)}+N(nt)=\frac{1}{\bar r}r_{N(nt)}\le \frac{1}{\bar r}(u+nt)<\frac{1}{\bar r}r_{N(nt)+1}=\tilde{r}_{N(nt)+1}+N(nt)+1.
\]
So
\[
N(nt)-\frac{nt}{\bar r}\le \frac{u}{\bar r}-\tilde{r}_{N(nt)}\le\frac{r_1}{\bar r}-\tilde{r}_{N(nt)}\le |\tilde{r}_{1}|+|\tilde{r}_{N(nt)}|+1,
\]
\[
N(nt)-\frac{nt}{\bar r}>\frac{u}{\bar r}-\tilde{r}_{N(nt)+1}-1>-\tilde{r}_{N(nt)+1}-1>-(|\tilde{r}_{N(nt)+1}|+1).
\]
We note that $N(t)\le [t/ \inf r]+1$ for $t>0$. Then there exists a constant $\kappa\in \N$ such that $N(nt)\le \kappa n$ for all $t\in[0,1]$. Hence
\[
\bigg|N(nt)-\big[\frac{nt}{\bar r}\big]\bigg|\le \bigg|N(nt)-\frac{nt}{\bar r}\bigg|+1\le 2\max_{1\le j\le \kappa n+1}\big|\tilde{r}_{j}\big|+2.
\]
So we have
\begin{align*}
&\bigg\|\sup_{t\in[0,1]}\big|N(nt)-[\frac{nt}{\bar r}]\big|\bigg\|_{L^{2(p-1)}(M^r)}\le2\bigg\|\max_{1\le j\le \kappa n+1}\big|\tilde{r}_{j}\big|\bigg\|_{L^{2(p-1)}(M^r)}+2\\
&=\frac{2}{\bar r}\bigg\|\max_{1\le j\le \kappa n+1}\big|r_{j}-j\bar r\big|\bigg\|_{L^{2(p-1)}(M^r)}+2\\
&\le C\bigg\|\max_{1\le j\le \kappa n+1}\big|r_{j}-j\bar r\big|\bigg\|_{L^{2(p-1)}(M)}\\
&\le Cn^{\frac{1}{2}},
\end{align*}
where the second inequality holds because we have $\|g\|_{L^{p}(M^r)}\le \bar{r}^{-1/p}|r|_{\infty}^{1/p}\|g\|_{L^{p}(M)}$
for a function $g(x,u)=g(x)$. And the last inequality is due to Lemma \ref{res}$(e)$, regarding $r-\bar r$ as an observable.
\end{proof}

\begin{lem}\label{wgn}
Let $p>2$. Then there exists a constant $C>0$ such that $\mathcal{W}_{p-1}(W_{n}, \frac{1}{\sqrt{\bar r}}\widetilde{W}_{[\frac{n}{\bar r}]})\leq Cn^{-\frac{1}{4}+\frac{1}{4(p-1)}}$ for all $n\ge 1$.
\end{lem}

\begin{proof}
Define $H(x,u)=\int_{0}^{u}v(x,s)\rmd s$ for $(x,u)\in M^r$. By (6.5) in \cite{AOP16},
\[
v_{nt}(x,u)=\tilde{v}_{N(nt)}(x)+H\circ\phi_{nt}(x,u)-H(x,u).
\]
Then $W_n(t)=\frac{1}{\sqrt{n}}\tilde{v}_{N(nt)}+F_n(t)$, where $\big|\sup_{t\in[0,1]}|F_n(t)|\big|_{\infty}=O(n^{-1/2})$.

By the definition of Wasserstein distance, we have
\begin{align*}
\mathcal{W}_{p-1}(W_{n}, \frac{1}{\sqrt{\bar r}}\widetilde{W}_{[\frac{n}{\bar r}]})&\leq\bigg\|\sup_{t\in[0,1]}\big|W_{n}(t)-\frac{1}{\sqrt{\bar r}}\widetilde{W}_{[\frac{n}{\bar r}]}(t)\big|\bigg\|_{L^{p-1}(M^r)}\\
&\leq\bigg\|\sup_{t\in[0,1]}\big|\frac{1}{\sqrt{n}}\tilde{v}_{N(nt)}-\frac{1}{\sqrt{\bar r}}\widetilde{W}_{[\frac{n}{\bar r}]}(t)\big|\bigg\|_{L^{p-1}(M^r)}+\bigg|\sup_{t\in[0,1]}|F_n(t)|\bigg|_{\infty}\\
&\le I_1+I_2+I_3+I_4+O(n^{-1/2}),
\end{align*}
where
\begin{align*}
&I_1=\bigg\|\sup_{t\in[0,1]}\big|\frac{1}{\sqrt{n}}\sum_{i=0}^{N(nt)-1}\tilde{v}\circ T^{i}-\frac{1}{\sqrt{n}}\sum_{i=0}^{[\frac{nt}{\bar r}]-1}\tilde{v}\circ T^{i}\big|\bigg\|_{L^{p-1}(M^r)},\\
&I_2=\bigg\|\sup_{t\in[0,1]}\big|\frac{1}{\sqrt{n}}\sum_{i=0}^{[\frac{nt}{\bar r}]-1}\tilde{v}\circ T^i-\frac{1}{\sqrt{n}}\sum_{i=0}^{[[\frac{n}{\bar r}]t]-1}\tilde{v}\circ T^{i}\big|\bigg\|_{L^{p-1}(M^r)}=O(n^{-1/2}),\\
&I_3=\bigg\|\sup_{t\in[0,1]}\big|\frac{1}{\sqrt{n}}\sum_{i=0}^{[[\frac{n}{\bar r}]t]-1}\tilde{v}\circ T^{i}-\frac{1}{\sqrt{\bar r}\sqrt{[\frac{n}{\bar r}}]}\sum_{i=0}^{[[\frac{n}{\bar r}]t]-1}\tilde{v}\circ T^{i}\big|\bigg\|_{L^{p-1}(M^r)},\\
&I_4=\bigg\|\sup_{t\in[0,1]}\big|\frac{1}{\sqrt{\bar r}\sqrt{[\frac{n}{\bar r}}]}\big(\big[\frac{n}{\bar r}\big]t-\big[\big[\frac{n}{\bar r}\big]t\big]\big)\tilde{v}\circ T^{[[\frac{n}{\bar r}]t]}\big|\bigg\|_{L^{p-1}(M^r)}=O(n^{-1/2}).
\end{align*}

It follows from a modification of Propositions \ref{est}, Proposition \ref{ntt} and the Cauchy-Schwarz inequality that
\begin{align*}
I_1=&\frac{1}{\sqrt{n}}\bigg\|\sup_{t\in[0,1]}\big|\sum_{i=[\frac{nt}{\bar r}]}^{N(nt)-1}\tilde{v}\circ T^{i}\big|\bigg\|_{L^{p-1}(M^r)}\\
\le& n^{-\frac{1}{2}}\bigg\|Z_{\kappa n}\big({(\kappa n)}^{-\frac{1}{2}}\sup_{t\in[0,1]}\big|N(nt)-\big[\frac{nt}{\bar r}\big]\big|+3\big)\bigg\|_{L^{p-1}(M^r)}\\
\le& n^{-\frac{1}{2}}\big\|Z_{\kappa n}\big\|_{{L^{2(p-1)}(M^r)}}\bigg({(\kappa n)}^{-\frac{1}{2}}\bigg\|\sup_{t\in[0,1]}\big|N(nt)-\big[\frac{nt}{\bar r}\big]\big|\bigg\|_{{L^{2(p-1)}(M^r)}}+3\bigg)\\
\le& Cn^{-\frac{1}{2}}\big\|Z_{\kappa n}\big\|_{{L^{2(p-1)}(M^r)}}\leq Cn^{-\frac{1}{4}+\frac{1}{4(p-1)}}.
\end{align*}
By Lemma \ref{res}$(e)$,
\begin{align*}
I_3\le & \bigg|\frac{1}{\sqrt{n}}-\frac{1}{\sqrt{\bar r}\sqrt{[\frac{n}{\bar r}}]}\bigg|\big[\frac{n}{\bar r}\big]^{\frac{1}{2}}\|\tilde v\|_\eta\\
\le & \frac{1}{\sqrt{\bar r}}\bigg|\frac{\sqrt{\frac{n}{\bar r}}-\sqrt{[\frac{n}{\bar r}]}}{\sqrt{\frac{n}{\bar r}\cdot[\frac{n}{\bar r}]}}\bigg| [\frac{n}{\bar r}\big]^{\frac{1}{2}}\|\tilde v\|_\eta\\
\le &Cn^{-\frac{1}{2}}\|\tilde v\|_\eta.
\end{align*}
Thus the result follows from these estimates.
\end{proof}

\begin{proof}[Proof of Theorem \ref{flow}]
Since $W_n(t)=\frac{1}{\sqrt{n}}\tilde{v}_{N(nt)}+F_n(t)$, where $\big|\sup_{t\in[0,1]}|F_n(t)|\big|_{\infty}=O(n^{-1/2})$, and $N(nt)\le \kappa n$ for $t\in[0,1]$. Then
\begin{align*}
\bigg\|\sup_{t\in[0,1]}|W_{n}(t)|\bigg\|_{L^{2(p-1)}(M^r)}\le \bigg\|\max_{j\le \kappa n}\big|\frac{1}{\sqrt{n}}\tilde{v}_{j}\big|\bigg\|_{L^{2(p-1)}(M^r)}+\bigg|\sup_{t\in[0,1]}|F_n(t)|\bigg|_{\infty}<\infty.
\end{align*}
So $W_n$ has a finite moment of order $2(p-1)$. Following the argument in the proof of Theorem \ref{wpc}, we can obtain the convergence result.
\end{proof}

\begin{proof}[Proof of Theorem \ref{nef}]
Consider the continuous process $\frac{1}{\sqrt{\bar r}}\widetilde{W}_{[\frac{n}{\bar r}]}\in C[0,1]$ with $\widetilde{W}_n$ defined in \eqref{dis}. Since $ W=_{d}\frac{1}{\sqrt{\bar r}}\widetilde W$, then for $p\ge 4$
\[
\mathcal{W}_{\frac{p}{2}}(W_{n}, W)\le \mathcal{W}_{\frac{p}{2}}(W_{n},\frac{1}{\sqrt{\bar r}}\widetilde{W}_{[\frac{n}{\bar r}]})+ \mathcal{W}_{\frac{p}{2}}(\frac{1}{\sqrt{\bar r}}\widetilde{W}_{[\frac{n}{\bar r}]},\frac{1}{\sqrt{\bar r}}\widetilde{W}).
\]
Note that no matter $\widetilde{W}_{[\frac{n}{\bar r}]}$ and $\widetilde{W}$ are regarded as processes defined on $(M^r, \mu^r)$ or $(M, \mu)$, the Wasserstein distance between their distributions on $C[0,1]$ remains the same.  By Lemma \ref{wgn} and Theorem \ref{thnon}, we can conclude that
$\mathcal{W}_{\frac{p}{2}}(W_{n}, W)\le Cn^{-\frac{1}{4}+\frac{1}{4(p-1)}}$.
\end{proof}

\subsection{Nonuniformly hyperbolic flows}
Suppose that $T:M\to M$ is a nonuniformly hyperbolic transformation of order $p>2$ as in Section 3.2. Given an integrable roof function $r:M\to \R^{+}$, as the way we define the semiflow in Section 5.1, we can define a suspension flow $\phi_t:M^r\to M^r$. Then we call $\phi_t:M^r\to M^r$ a nonuniformly hyperbolic flow. We assume that $\inf r>0$. Under the setting of Section 5.1, it follows from \cite[Corollary 3.5]{MR2175992} that the WIP holds for nonuniformly hyperbolic flows; that is, $W_n\to_{w}W$ in $C[0,1]$, where $W$ is a Brownian motion with mean zero and variance $\sigma^2\ge 0$.

\begin{thm}
Let $\phi_t:M^r\to M^r$ be a nonuniformly hyperbolic flow. Suppose that the return time $R\in L^p(Y)$ for $p> 2$ and the roof function $r:M\to \R^+$ is H$\ddot{o}$lder. Let $v:M^r\to \R$ be a H\"older observable with $\int_{M^r} v \rmd \mu^r=0$. Then $\mathcal{W}_{q}(W_{n},W)\to 0$ in $C[0,1]$ for all $1\le q< 2(p-1)$.
\end{thm}

\begin{proof}
This is almost identical to the proof of Theorem \ref{flow}.
\end{proof}

\begin{thm}\label{flow2}
Let $\phi_t:M^r\to M^r$ be a nonuniformly hyperbolic flow. Suppose that the return time $R\in L^p(Y)$ for $p\ge4$ and the roof function $r:M\to \R^+$ is H$\ddot{o}$lder. Let $v:M^r\to \R$ be a H\"older observable with $\int_{M^r} v \rmd \mu^r=0$. Then there exists a constant $C>0$ such that $\mathcal{W}_{\frac{p}{2}}(W_{n},W)\leq Cn^{-\frac{1}{4}+\frac{1}{4(p-1)}}$ for all $n\geq 1$.
\end{thm}

\begin{proof}
Similar to the proof of Theorem \ref{nef}, the result follows from Theorem \ref{hyp}.
\end{proof}

\begin{exam}[Planar periodic Lorentz gases]
The $2$-dimensional periodic Lorentz gas is a model of electron gases in metals introduced by Sina{\u\i} \cite{S70}. The Lorentz flow is a billiard flow on $\Omega={\mathbb T}^2-\cup_{i=1}^{k}\Omega_i$, where $\Omega_i$'s are disjoint convex regions with $C^3$ boundaries. The Poincar\'{e} map $T$ is defined on $M=\partial\Omega\times[-\frac{\pi}{2},\frac{\pi}{2}]$. Under the finite horizon condition, which means that the time between collisions is uniformly bounded, Bunimovich et al \cite{Bu91} proved the WIP for the Lorentz flow and Young \cite{MR1637655} demonstrated that the Poincar\'{e} map has exponential decay of correlations. By Theorem \ref{flow2}, we can obtain the convergence rate for the Lorentz flow:  for all $p\geq 4$, $\mathcal{W}_{\frac{p}{2}}(W_{n},W)\leq Cn^{-\frac{1}{4}+\frac{1}{4(p-1)}}$.
\end{exam}

\begin{exam}[Dispersing billiards with cusps]
Consider dispersing billiards with cusps, where the boundary curves are all dispersing but the interior angles at corner points are zero. Since there exists an alternative cross-section $M'\subset M$ such that the Poincar\'{e} map $T':M'\to M'$ is modelled by a Younger tower with exponential tails. B\'{a}lint and Melbourne \cite{BM2008} proved the exponential mixing rate and the ASIP for the flow corresponding to the dispersing billiards with cusps. So by Theorem \ref{flow2} we can obtain the convergence rate: for all $p\geq 4$, $\mathcal{W}_{\frac{p}{2}}(W_{n},W)\leq Cn^{-\frac{1}{4}+\frac{1}{4(p-1)}}$.
\end{exam}

\section{Application to homogenization problem}
\setcounter{equation}{0}

We consider fast-slow systems of the discrete form
\begin{equation}\label{xxn}
x_\epsilon(n+1)=x_\epsilon(n)+\epsilon^2g(x_\epsilon(n),y(n),\eps)+\epsilon h(x_\epsilon(n))v(y(n)), ~x_\epsilon(0)=\xi,
\end{equation}
where $g:\mathbb{R}\times M\times \R^+\rightarrow\mathbb{R}$, $h:\mathbb{R}\rightarrow\mathbb{R}$ satisfy some regularity conditions and $v\in C^\eta(M)$ with $\int_M v\rm d\mu=0$. The fast variables $y(n)\in M$ are generated by iterating a nonuniformly expanding/hyperbolic transformation, that is  $y(n+1)=Ty(n)$, $y(0)=y_0$. Here $T:M\to M$ satisfies the setting in Section 3. The initial condition $\xi\in \R$ is fixed and $y_0\in M$ is chosen randomly, which is the reason for the emergence of randomness from deterministic dynamical systems.\\
{\bf Regularity conditions:}\\
(1) $ g:\mathbb{R}\times M\times \R^+\rightarrow\mathbb{R}$ is bounded.\\
(2) $ g(x,y,0)$ is Lipschitz in $x$ uniformly in $y$, that is $|g(x_1,y,0)-g(x_2,y,0)|\le $Lip$g|x_1-x_2|$, for all $x_1,x_2\in \mathbb{R}, y\in M$.\\
$(3)$ $\sup_{x\in \R}\sup_{y\in M}|g(x,y,\epsilon)-g(x,y,0)|\le C\epsilon^{\frac{1}{4}}$.\\
$(4)$ $ g(x,y,0)$ is H\"older continuous in $y$ uniformly in $x$, i.e. $\sup_{x\in \R}|g(x,\cdot,0)|_\eta<\infty$.\\
$(5)$ $h$ is exact; that is $h=\frac{1}{\psi'}$, where $\psi$ is a monotone differentiable function and $\psi'$ denotes the derived function. Moreover, $h, h', \frac{1}{h}$ are bounded.

Let ${\hat x}_\epsilon(t)=x_\epsilon(t\epsilon^{-2})$ for $t=0, \epsilon^2, 2\epsilon^2,\ldots$, and linearly interpolate to obtain ${\hat x}_\epsilon\in C[0,1]$. Then it follows from \cite[Theorem 1.3]{MR3064670} that ${\hat x}_\epsilon\rightarrow_{w} X$ in $C[0,1]$, where $X$ is the solution to the Stratonovich SDE
\begin{equation}\label{sde}
\mathrm{d}X=\left\{ \bar{g}(X)-\frac{1}{2}h(X)h'(X)\int_{M}v^{2}\mathrm{d}\mu\right\}\mathrm{d}t+h(X)\circ \mathrm{d}W, ~X(0)=\xi.
\end{equation}
Here $W$ is a Brownian motion with mean zero and variance $\sigma^2$, $\bar{g}(x)=\int_{M}g(x,y,0)\mathrm{d}\mu(y)$.

Define $W_\epsilon(t)=\epsilon\sum_{j=0}^{t\epsilon^{-2}-1}v(y(j))$ for $t=0, \epsilon^2, 2\epsilon^2,\ldots$, and linearly interpolate to obtain
$W_\epsilon\in C[0,1]$. Comparing $W_\epsilon$ with $W_n$, we can see that $\epsilon$ is identified with $n^{-\frac{1}{2}}$. Hence it follows from Theorems \ref{thnon} and \ref{hyp} that $\mathcal{W}_2(W_\epsilon,W)=O(\epsilon^{\frac{p-2}{2(p-1)}})$.

\begin{thm}\label{hom}
Let $T:M\rightarrow M$ be a nonuniformly expanding/hyperbolic transformation of order $p\ge 4$. Suppose that the regularity conditions hold. Then there exists a constant $C>0$ such that $\mathcal{W}_{2}({\hat x}_\epsilon, X)\le C\epsilon^{\frac{p-2}{4p}}$.
\end{thm}

\begin{proof}
First, suppose that $h(x)\equiv 1$ and $M=[\eps^{-\frac{4}{3}}]$. By \cite[Proposition 5.4]{MR3975864}, we can write
\[
{\hat x}_\epsilon(t)=\xi+W_\epsilon(t)+D_\epsilon(t)+E_\eps(t)+\int_{0}^{t}\bar{g}({\hat x}_\epsilon(s))\rmd s,
\]
where
\[
D_\epsilon(t)=\epsilon^{\frac{2}{3}}\sum_{n=0}^{[t\epsilon^{-\frac{2}{3}}]-1}J_{\eps}(n), \quad J_\epsilon(n)=\epsilon^{\frac{4}{3}}\sum_{j=nM}^{(n+1)M-1}\tilde {g}(x_\eps(nM),y(j)), \quad
\tilde {g}(x,y)=g(x,y,0)-\bar g(x)
\]
and
\begin{equation}\label{Eeps}
\bigg\|\sup_{t\in[0,1]}|E_\eps(t)|\bigg\|_{L^2}\le C\eps^{\frac{1}{4}}.
\end{equation}

Let $B_\eps(R_\eps)=\left\{\sup_{t\in[0,1]}|\hat{x}_{\eps}(t)|\le R_\eps\right\}$, where $R_\eps =(-32\sigma^2\log \eps)^{\frac{1}{2}}$. By \cite[Lemma 5.5]{MR3975864}, we have
\[
\mu(\sup_{t\in[0,1]}|\hat{x}_{\eps}(t)|\ge R_\eps)\le C\eps^{\frac{p-2}{2p}}.
\]
Since
\[
\mu(\sup_{t\in[0,1]}|D_\eps 1_{B^c_\eps(R_\eps)}(t)|>0)\le \mu(\sup_{t\in[0,1]}|\hat{x}_{\eps}(t)|\ge R_\eps)\le C\eps^{\frac{p-2}{2p}},
\]
and $D_\eps$ is bounded, we get
\[
\bigg\|\sup_{t\in[0,1]}|D_\eps 1_{B^c_\eps(R_\eps)}(t)|\bigg\|_{L^2}\le C\eps^{\frac{p-2}{4p}}.
\]
It follows from \cite[Lemma 5.6]{MR3975864} that $\big\|\sup_{t\in[0,1]}|D_\eps 1_{B_\eps(R_\eps)}(t)|\big\|_{L^2}\le C\eps^{\frac{1}{3}}(-\log\eps)^{\frac{1}{4}}$. Hence
\begin{equation}\label{Deps}
\bigg\|\sup_{t\in[0,1]}|D_\eps(t)|\bigg\|_{L^2}\le C\eps^{\frac{p-2}{4p}}.
\end{equation}

Next, define a continuous map $G:C[0,1]\rightarrow C[0,1]$ as $G(u)=v$, where $v$ is the unique solution to $v(t)=\xi+u(t)+\int_{0}^{t}\bar{g}(v(s))\rmd s$. Since $\bar{g}$ is Lipschitz, according to the existence and uniqueness of solutions to ordinary differential equations, $G$ is well-defined. By Gronwall's inequality, $G$ is Lipschitz with Lip$G\le e^{Lip \bar g}$.

Since $X=G(W)$ and ${\hat x}_\epsilon=G(W_\eps+D_\eps+E_\eps)$, we have
\[
\mathcal{W}_{2}({\hat x}_\epsilon, X)=\mathcal{W}_{2}(G(W_\eps+D_\eps+E_\eps), G(W))\le e^{Lip \bar g}\mathcal{W}_{2}(W_\eps +D_\eps+E_\eps, W).
\]
Following $\mathcal{W}_2(W_\epsilon,W)=O(\epsilon^{\frac{p-2}{2(p-1)}})$ and the above estimates \eqref{Eeps}--\eqref{Deps}, we obtain the conclusion in the case $h\equiv 1$.

When $h\not\equiv 1$, by a change in variables, $z_\eps(n)=\psi(x_\eps(n))$, ${\hat z}_\epsilon(t)=\psi({\hat x}_\epsilon(t))$, we can
reduce the case of multiplicative noise to the case of additive noise. That is,
\begin{align*}
z_\eps(n+1)-z_\eps(n)=\eps v(y(n))+\eps^2G(z_\eps(n),y(n),\eps),
\end{align*}
where $G(z,y,\eps):=\psi'(\psi^{-1}z)g(\psi^{-1}z,y,\eps)+\frac{1}{2}\psi''(\psi^{-1}z) (\psi'(\psi^{-1}z))^{-2}v^2(y)+O(\eps)$; see \cite{MR3064670, MR3795019} for the calculations. Moreover, we can verify that $G(z,y,\eps)$ satisfies the regularity conditions (1)--(4).

Let
\begin{align*}
\bar{G}(z):=\psi'(\psi^{-1}z)\bar{g}(\psi^{-1}(z))+\frac{1}{2}\psi''(\psi^{-1}(z)) (\psi'(\psi^{-1}(z)))^{-2}\int_{M}v^2\rm d\mu.
\end{align*}
Consider the SDE
\begin{equation}\label{zwt}
\rmd Z=\rmd W+\bar{G}(Z)\rmd t, \quad Z(0)=\psi(\xi).
\end{equation}
Then ${\hat z}_\eps\rightarrow_{w}Z$, where $Z$ is the solution to \eqref{zwt}, and $\mathcal{W}_{2}({\hat z}_\eps,Z)=O(\eps^{\frac{p-2}{4p}})$. Because the Stratonovich integral satisfies the usual chain rule, we can see that $Z=\psi(X)$ satisfies the SDE \eqref{zwt} as in \cite{MR3064670}. Hence we have
\[
\mathcal{W}_{2}({\hat x}_\epsilon, X)=\mathcal{W}_{2}(\psi^{-1}({\hat z}_\eps),\psi^{-1}(Z))\le \text{Lip}(\psi^{-1})\mathcal{W}_{2}({\hat z}_\eps,Z)=O(\eps^{\frac{p-2}{4p}}).
\]
The proof is complete.
\end{proof}

\begin{rem}
By minor modifications at several places in the proof of Theorem \ref{hom}, we can get the convergence rate of $\mathcal{W}_q(\hat{x}_\eps, X)$ for all $1<q\le p/2$. But $\mathcal{W}_1$ and $\mathcal{W}_2$ are the ``most useful" Wasserstein distances in Villani's words (see \cite[Remark 6.6]{MR2459454}), so we compute only the $\mathcal{W}_2$-rate for illustration in this application.
\end{rem}

\begin{rem}\label{diss}
In this application, we obtain the Wasserstein convergence rate $O(\eps^{\frac{1}{4}-\delta})$ as $p\to\infty$ under the condition
$|g(x,y,\epsilon)-g(x,y,0)|_\infty\le C\epsilon^{\frac{1}{4}}$.\ While in \cite{MR3975864}, Antoniou and Melbourne obtained the weak convergence rate $O(\eps^{\frac{1}{3}-\delta})$ as $p\to\infty$  under the condition $|g(x,y,\epsilon)-g(x,y,0)|_\infty\le C\epsilon^{\frac{1}{3}}$.
If they weaken the condition to our case, they can only obtain a rate like ours, i.e. $O(\eps^{\frac{1}{4}-\delta})$; but even if we strengthen our condition to their case, we still can only obtain the rate $O(\eps^{\frac{1}{4}-\delta})$, instead of the expected $O(\eps^{\frac{1}{3}-\delta})$.
\end{rem}

\appendix

\section{}

%

\begin{thm}[Skorokhod embedding theorem \cite{MR624435}]\label{ske}
Let $\{S_n=\sum_{i=1}^{n}X_{i},\mathcal{F}_{n}, n\ge 1\}$ be a zero-mean, square-integrable martingale. Then there exist a probability space supporting a (standard) Brownian motion $W$ and a sequence of nonnegative variables $\tau_{1}, \tau_{2},\ldots$ with the following properties: if $T_{n}=\sum_{i=1}^{n}\tau_{i}$, $S'_{n}=W(T_n)$, $X'_1=S'_1$, $X'_n=S'_{n}-S'_{n-1}$ for $n\ge 2$, and $\mathcal{B}_{n}$ is the $\sigma$-field generated by $S'_1,\ldots, S'_{n}$ and $W(t)$ for $0\le t\le T_n$, then
\begin{enumerate}
\item $\{S_{n}, n\ge 1\}=_{d} \{S'_n, n\ge 1\}$;
\item $T_n$ is a stopping time with respect to $\mathcal{B}_n$;
\item $E(\tau_{n}|\mathcal{B}_{n-1})=E(|X'_{n}|^{2}|\mathcal{B}_{n-1}) $ a.s.;
\item for any $p>1$, there exists a constant $C_p<\infty$ depending only on $p$ such that
\[
E(\tau_{n}^{p}|\mathcal{B}_{n-1})\leq C_{p}E(|X'_{n}|^{2p}|\mathcal{B}_{n-1})=C_{p}E(|X'_{n}|^{2p}|X'_1,\ldots,X'_{n-1}) \quad \hbox{a.s.},
\]
where $C_p=2(8/\pi^2)^{p-1}\Gamma(p+1)$, with $\Gamma$ being the usual Gamma function.
\end{enumerate}
\end{thm}



\begin{thm}\footnote{This estimate was suggested to us by Prof. Ian Melbourne.}\label{mpe}
Let $X_1, X_2,\ldots$ be a sequence of identically distributed random variables with $\|X_1\|_{L^p}<\infty$. Then $\big\|\max_{1\le k\le n}| X_k|\big\|_{L^{p}}=o(n^{\frac{1}{p}})$ as $n\to \infty$.
\end{thm}

\begin{proof}
For $\epsilon>0$, we have
\[
\big|X_k\big|^p\le n\epsilon+\big|X_k\big|^p1_{\{|X_k|^p>n\epsilon\}}.
\]
So
\[
\max_{1\le k\le n}\big|X_k\big|^p\le n\epsilon+\sum_{k=1}^{n}\big|X_k\big|^p1_{\{|X_k|^p>n\epsilon\}}.
\]
Since $\{X_k\}$ is identially distributed,
\[
\E\max_{1\le k\le n}\big|X_k\big|^p\le n\epsilon+n\E\big[|X_k|^p1_{\{|X_k|^p>n\epsilon\}}\big].
\]
It follows that
\[
\frac{1}{n}\E\max_{1\le k\le n}\big|X_k\big|^p\le \epsilon+\E\big[|X_k|^p1_{\{|X_k|^p>n\epsilon\}}\big]\to \epsilon
\]
as $n\to \infty$. Hence the result follows because $\epsilon$ can be taken arbitrarily small.
\end{proof}

\section*{Acknowledgements}
The authors sincerely thank Professor Ian Melbourne for his valuable suggestions. The authors are deeply grateful to the referee for his/her
great patience and very careful reading of the paper and for many valuable suggestions which lead to significant improvements of the paper.
This work is supported by NSFC Grants 11871132, 11925102, Dalian High-level Talent Innovation Project (Grant 2020RD09), and
Xinghai Jieqing fund from Dalian University of Technology.

\end{document}